\documentclass{article}
\usepackage{amsmath}
\usepackage{amssymb}
\usepackage{amsfonts}
\usepackage{a4wide}
\usepackage{amsthm}
\def\Z{\mathbb Z}

\def\mod{\text{ mod }}

\newcounter{NN}
\newtheorem{theorem}[NN]{Theorem}

\newtheorem{conjecture}[NN]{Conjecture}
\newtheorem{lemma}[NN]{Lemma}

\begin{document}
\bibliographystyle{plain}
\title{Somos-4 and Somos-5 are arithmetic divisibility sequences}
\author{Peter H.~van der Kamp}
\date{Department of Mathematics and Statistics\\
La Trobe University, Victoria 3086, Australia\\[5mm]
\today
}

\maketitle

\begin{abstract}
We provide an elementary proof to a conjecture by Robinson that multiples of (powers of)
primes in the Somos-4 sequence are equally spaced. We also show, almost as a corollary, for
the generalised Somos-4 sequence defined by $\tau_{n+2}\tau_{n-2}=\alpha\tau_{n+1}\tau_{n-1}+\beta\tau_n^2$ and initial values $\tau_1=\tau_2=\tau_3=\tau_4=1$, that the polynomial $\tau_n(\alpha,\beta)$ is a divisor of $\tau_{n+k(2n-5)}(\alpha,\beta)$ for all $n,k\in\Z$ and establish a similar result for the generalized Somos-5 sequence.
\end{abstract}

\section{Introduction} \label{s1}
An arithmetic divisibility sequence is neither a divisibility sequence,
nor an arithmetic sequence. We define it to be a sequence $\{f_n\}_{n\in\Z}$
such that there is a function $d:\mathbb Z \mapsto \mathbb{N}$, called the
common difference function, such that
\[
\forall n\in \Z: d(n) \mid (m-n) \implies f_n \mid f_m.
\]
Clearly, if for some sequence $\{f_n\}_{n\in\Z}$ prime powers appear in arithmetic
progressions that sequence is an arithmetic divisibility sequence. Furthermore, we
then have the following nice property: for all $m=n+k d(n)$ and $r=m+s d(m)$, with
$n,k,s\in\Z$, there is an $l\in\Z$ such that $r=n + l d(n)$. Primes do not have to appear
in arithmetic progressions, e.g. consider the periodic
arithmetic divisibility sequence
\[
f_n=\begin{cases}
12 & \text{if } n\equiv 0 \mod 6,\\
1 & \text{if } n\equiv 1,5 \mod 6,\\
6 & \text{if } n \equiv 2,4 \mod 6,\\
4 & \text{if } n\equiv 3 \mod 6.
\end{cases}
\]
where we have even made sure that $f_n \mid f_m \implies d(n) \mid (m-n)$.
The above mentioned nice property also holds if the common difference function is polynomial.
This follows from the fact that $p(x)$ divides $p(x+p(x)q(x))$ for all polynomials $p(x),q(x)$, which can be seen by substituting $y=x+p(x)q(x)$ in $y-x \mid p(y)-p(x)$.

Sequences $\{\tau_n\}$ defined by
\begin{equation} \label{Sok}
\tau_{n+k-2}\tau_{n-2}=\alpha\tau_{n+k-3}\tau_{n-1}+\beta\tau_{n+k-4}\tau_n
\end{equation}
possess the Laurent property, that is, all their terms are Laurent polynomials in their initial values.
The equations (\ref{Sok}) arise as a special case of the Hirota-Miwa equation, whose
integrability condition is equivalent to Laurentness \cite{Mas}. For $k=4$ and $k=5$ these
equations are called Somos-4 and Somos-5. 
The Laurent property implies in particular that if the initial values are $1$, and $\alpha,\beta$ are integers, both recurrences define integer sequences. Considering $\alpha$ and $\beta$ to be variables, we obtain polynomial sequences.
The Laurent property was introduced by Hickerson, cf. \cite{Rob}, to prove the integrality of the next sequence, Somos-6, which does not have the form (\ref{Sok}) but is a special case of the discrete BKP equation \cite{Som,Mas}. Initially the integrality, and secondly the Laurent property itself, seemed to be curious properties \cite{Gal}. By now there is extended literature on equations that possess the Laurent property, have Laurentness, or display the Laurent phenomenon,
\cite{ACH,FoZe,FM,HoneLP,HS,HW,LP,Mas,HK}, and the property is well understood in terms of cluster algebras \cite{FoZeC}, which have had a profound impact in diverse areas of mathematics, cf. \cite{Fom}. The Somos-$k=4,5$ sequences also satisfy a strong Laurent property, e.g. for Somos-4 the terms are Laurent in $\tau_1$, polynomial in $\tau_2,\ldots,\tau_k$, and polynomial in $\cal I$, an additional (invariant) rational function of the initial values \cite{HS}. Therefore, integral sequences may be obtained for other than unit initial values. The Somos-4 sequence with $\alpha=-1$, $\beta=2$ and initial values $\tau_1=\tau_2=1,\ \tau_3=2,\ \tau_4=3$ is an integer sequence which extends the
Fibonacci numbers; we have $\tau_n=(-1)^{n+1}\tau_{-n}$ and $\tau_{n}=\tau_{n-1}+\tau_{n-2}$.

Few results on divisibility properties of {\em the} Somos-($k=4,5$) sequences (with $\alpha=\beta=1$ and $\tau_1=\cdots=\tau_k=1$) are known.
In \cite{EMW} it is shown that every term beyond the fourth of the Somos-4 sequence has a primitive divisor, i.e. a prime which does not divide any preceding term. Jones and Rouse \cite{JR} show that the density of primes dividing at least one term of the Somos-4 sequence is 11/21. In \cite{KMMT} the authors establish that the terms of Somos-4 (with $\alpha=\beta=1$) are irreducible Laurent polynomials in their initial values and pairwise co-prime. Robinson \cite{Rob} showed that, for $k=4,5$, the $i$-th and $j$-th terms of the Somos-$k$  sequence are relatively prime whenever $| i-j |\leq k$. He infers that both sequences are periodic modulo $m$ for every $m\in\mathbb N$. In \cite[section 4]{Rob} he presents results about Somos-4 and Somos-5 obtained by calculation, for which he did not have general proofs. One observation he made is that if a prime $p$ divides any term of Somos-$k$, then the multiples of $p$ are equally spaced. In proof he added that Clifford S. Gardner proved this result, but a reference was not provided. This conjecture and generalisations have been proven by C.S. Swart in her thesis \cite{Swa}. She also proved, for more general Somos-4 sequences, that either all multiples of a prime $p$ are divisible by exactly the same power of $p$, or $\exists r\in\mathbb{N} \forall k\geq r \in\mathbb N$ some term is divisible by $p^k$ exactly, cf. section \ref{s7}. The following problem is still open: to prove that for particular Somos sequences for all but finitely many primes $p$ the multiples of $p$ are not divisible by exactly the same power of $p$.

In this short paper we provide a simple proof, in section \ref{s4}, that for all primes $p\in\mathbb P$ the set $\{n\in\mathbb Z: p^k\mid \tau_n\}$ has either less than 2 elements, or its elements form a complete arithmetic sequence. This also holds for rational Somos-$k=4,5$ sequences, for primes that do not appear in the denominator of the invariant function, as long as the initial conditions are pairwise co-prime. It also holds true in elliptic divisibility sequences, see section \ref{s3}.
We employ an elementary result from number theory, which we haven't found in the literature, namely that the terms of a subset $S\subset\mathbb{Z}$ with the property that $2s-t\in S$ for all $s,t\in S$, form an arithmetic sequence. A proof of this result is presented in section \ref{s2}. Combining the result on the appearance of primes with a symmetry of the equation and taking unit initial values we find Somos-$k=4,5$ sequences which are arithmetic divisibility sequences with common difference function $d_k(n)=2n-k-1$, i.e. we prove the polynomial divisibility
\[
\tau_n(\alpha,\beta) \mid \tau_{n+ld_k(n)}(\alpha,\beta) \quad \forall n,l\in\Z.
\]
We briefly investigate different initial values which give rise to Laurent sequences with the same divisibility property, as well as initial values that give rise to polynomial sequences with slightly different divisibility properties. We note that if one could prove for odd primes $p$ that the multiples of
$p$ in {\it the} Somos-4 \& 5 sequences are not divisible by exactly the same power of $p$ then we
would have the following.
\begin{conjecture} \label{conj}
Let $\tau_n$ be the terms of the Somos-$k$ sequence with $k=4,5$, $\alpha=\beta=1$, $\tau_1=\cdots=\tau_k=1$, $d=2n-k-1$ and
$q=\tau_n$ when $(k+1)\nmid n$ or $q=\tau_n/2$ when $(k+1)\mid n$. Then
\[
q^{m+1}\mid \tau_l \Leftrightarrow l = n + (\frac{q^m-1}{2}+kq^m) d.
\]
\end{conjecture}
A similar result for the extended Fibonacci sequence would provide an alternative proof for the fact that the $(m+1)$-st power of the $n$-th Fibonnaci number divides the $(nf_n^m)$-th Fibonnaci number, as was shown by M. Cavachi \cite{Cav}, at the age of 16. We note the case $n=3$ is exceptional; here $2^{m+2}\mid f_{2^m3}$.


\section{Sets of differences} \label{s2}
In \cite{Beu} a {\bf set of differences} is defined to be a subset $S$ of $\mathbb{N}$ such that $s,t\in S,\ s>t \Rightarrow s-t\in S$. A useful result is that the elements in sets of differences are multiples of the smallest element. Extending the definition to subsets of $\mathbb Z$ a similar statement may be used to provide a quick proof that elliptic divisibility sequences are divisibility sequences. 

For our purposes we define a subset $S$ of $\mathbb{Z}$ to be a {\bf modified set of differences} if $s,t\in S\Rightarrow 2s-t\in S$.
\begin{lemma} \label{lemma}
A modified set of differences has less than two elements, or its terms form a complete arithmetic sequence.
\end{lemma}
\begin{proof}
Let $s,t$ be elements of a modified set of differences $V$. We first show, by induction,
that
\begin{equation}\label{st}
(k+1)s-kt\in V, \text{ and } (k+1)t-ks\in V,
\end{equation}
for all $k\in \mathbb Z$. By definition it is true for $k=1$. Assuming the statement we obtain that
$2s-((k+1)s-kt)=kt-(k-1)s$ and $2t-((k+1)s-kt)=(k+2)t-(k+1)s$ are in $V$, as well as the
same expressions with $s,t$ interchanged.

If $s=t$ all the above elements coincide, the set may only have one element. If $s$ and
$t$ are distinct we may take $s<t$ in which case $(k+1)t-ks$ with $k\in\mathbb N$ provides
plenty of positive elements. Let $a,b>a$ be the smallest and the next smallest positive elements of $V$ and denote $d=b-a$. As $a-d=2a-b\in V$ and $a$ is the smallest positive element we have
$a<d$. Taking $s=a$ and $t=2a-b$ in the first element of (\ref{st}) we find that $kd+a\in V$
for all $k\in\mathbb Z$. We show the converse, $s\in V \Rightarrow d \mid (s-a)$.

For any $s\in V$ we may write $s=qd+r$ with $0\leq r<d$. If $r=a$ then we are done.
Let $t=qd+a\in V$. Assume $r<a$. The element $z=(k+1)s-kt=s-k(a-r)$, with $k=\lfloor \frac{s}{a-r}\rfloor$, is in $V$ and $0\leq z < a-r$ which is impossible, as $a$ is the smallest positive element of $V$. Assuming $r>a$, we have, with $k=\lfloor \frac{t}{r-a}\rfloor$, that $p=(k+1)t-ks=t-k(r-a)$ is in $V$ and $0\leq p< r-a$. We also have, replacing $k$ with $k-1$, that $q=p+(r-a)\in V$. As $p<d<b$ we need $p=a$. But with $p=a$ we have $a<q=r<d<b$ which contradicts that $b$ is the second smallest positive element.  
\end{proof}

\section{Elliptic divisibility sequences, and integrality} \label{s3}
Elliptic divisibility sequences (EDS) were introduced by Morgan Ward \cite{War,Ward} as sequences of integers $\{a_n\}_{n=0}^\infty$, that satisfy, for all $m\geq n\geq 1$,
\begin{equation} \label{for}
a_{m+n}a_{m-n}=\left|
\begin{matrix}
a_na_{m-1} & a_{n-1}a_{m} \\
a_{n+1}a_{m} & a_na_{m+1} 
\end{matrix}\right|
\end{equation}
and $n\mid m\Rightarrow a_n\mid a_m$. He calls a sequence proper if $a_0=0, a_1=1, a_2^2+a_3^2\neq 0$, and shows that a proper solution to (\ref{for}) is an EDS if and
only if $a_2,a_3,a_4$ are integers and $a_2\mid a_4$. Ward first shows by induction, that all terms are integers, and then by another induction step, the divisibility property. 

Ward's proof of integrality is generalised by Hone and Swart \cite{HS}, who proved the following strong Laurent property for Somos-4: the terms $\tau_n$ are polynomials in $\alpha,\beta,\tau_1^{\pm1},\tau_2,\tau_3,\tau_4$, and $\mathcal I$, where
\begin{equation} \label{I}
\mathcal I = \alpha^2+\beta T,\quad
T=\frac{\tau_{1}^2\tau_{4}^2+\alpha(\tau_{2}^3\tau_{4}+\tau_{1}\tau_{3}^3)+\beta\tau_{2}^2\tau_{3}^2}{\tau_{1}\tau_{2}\tau_{3}\tau_{4}}
\end{equation}
Therefore, if $\tau_1=\pm 1$ and $\mathcal I$ is integer, the sequence consist of integers.
We remark that taking $n=2$ in (\ref{for}) gives us a special Somos-4 sequence,
\begin{equation} \label{eds}
a_{m+2}a_{m-2}=a_2^2a_{m+1}a_{m-1}-a_1a_3a_m^2.
\end{equation}
Taking $\alpha=a_2^2,\beta=-a_1a_3$, $\{\tau_i=a_i\}_{i=1}^4$ and $a_1^2=1$ in (\ref{I}) we find $\mathcal I=-a_4/a_2$, whose integrality implies that the sequence $\{a_m\}$ consists of integers.

For Somos-5 we have the following strong Laurent property \cite[Theorem 3.7]{HS}:
The terms $\tau_{n>0}$ of Somos-5 are polynomial in $\alpha,\beta,\tau_1^{\pm1},\tau_2^{\pm1},\tau_3,\tau_4,\tau_5$,
and $\mathcal J$, where
\[
\mathcal J = \beta + \alpha S,\quad
S=\frac{(\tau_{1}\tau_{5}+\alpha\tau_{3}^2)(\tau_1\tau_4^2+\tau_{2}^2\tau_5)+
\beta\tau_{2}\tau_{3}^3\tau_4}{\tau_{1}\tau_{2}\tau_{3}\tau_{4}\tau_5}.
\]
The sequence is an integral sequence if $\tau_1,\tau_2\in\{\pm 1\}$, $\mathcal I$ is integer and
one can find three consecutive integers preceding two units.

An EDS uniquely extends to a sequence over $\Z$. Taking $m=2$ in (\ref{eds}) one finds
$a_0=0$. Taking $m=1$ in (\ref{eds}) we find $a_{3}a_{-1}=-a_1a_3 a_1^2\Rightarrow a_{-1}=-1$ (assuming $a_3\neq 0$). Taking $m=-1$ and $n=k-1$ in (\ref{for}) gives
$a_{k-2}a_{-k} = - a_{k-2}a_{k}a_{-1}^2 \Rightarrow a_{-k}=-a_k$ (assuming $a_{k-2} \neq 0$). For $k=2$ this doesn't work, but there exist a determining equation, e.g. take $m=0$ in (\ref{eds}). One way to deal with the occurrence of zeros other than $a_0$ is to take the initial values as parameters (let $a_4$ be a multiple of $a_2$), generate a polynomial sequence,
and then specialise. See \cite[Appendix A]{HS} for another discussion on zeros.

Taking the initial values as parameters (or assuming them to be co-prime) one can inductively show that gcd($a_n,a_{n+1}$)=1 for all $n>1$. The divisibility property follows by showing that $V_k=\{m \in \mathbb Z : a_k \mid a_m\}$, with $k>1$, is a set of differences with at least two elements. To show this we employ a second family of recurrences
\begin{equation} \label{fora2}
a_1a_2a_{m+n+1}a_{m-n}=\left|
\begin{matrix}
a_na_{m-1} & a_{n-1}a_{m} \\
a_{n+2}a_{m+1} & a_{n+1}a_{m+2}
\end{matrix}\right| .
\end{equation}
Taking $m=s$ and $n=t-s$ in 
(\ref{for}) gives
\[
a_{t}a_{2s-t}=a_{t-s}^2a_{s-1}a_{s+1}-a_{t-s-1}a_{t-s+1}a_{s}^2.
\]
If $s\in V_k$ then $s\pm1\not\in V_k$, and thus $a_k\mid a_{t-s}^2$.
Taking $m=s$ and $n=t-s-1$ in (\ref{fora2}) gives
\[
a_1a_2a_{t}a_{2s-t+1}=a_{t-s-1}a_{t-s}a_{s-1}a_{s+2}-a_{t-s-2}a_{t-s+1}a_{s}a_{s+1},
\]
and so $a_k\mid a_{t-s-1}a_{t-s}$. Together with $a_{t-s-1}$ and $a_{t-s}$ being co-prime we find that $s,t\in V_k\Rightarrow s-t\in V_k$. As both $0$ and $k$ are elements of $V_k$ we have $V_k=k\mathbb Z$. Thus, the polynomial sequence $\{a_n\}$ is an arithmetic divisibility sequence with common difference function
\[
d(n)=\begin{cases} n & n \neq 0, \\ \infty & n=0. \end{cases}
\]
For special initial values the value of $d$ at 0 can be finite.

Ward does not give much detail about the consistency of the family of recurrences
\ref{for}. He does however provide an explicit solution for $a_n$ in terms of the
Weierstrass sigma function, and both (\ref{for}) and (\ref{fora2}) are direct consequences 
of the corresponding three-term relation, cf. \cite{Hone5}. For an algebraic approach
we refer the reader to \cite{PooSwa}.

The fact that an EDS $\{a_n\}$ is a divisibility sequence does not imply that
multiples of powers of primes are equally spaced. However, for an EDS whose
initial values satisfy gcd($a_n,a_{n+1}$)=1 the above argument shows that
$V=\{n \in \mathbb Z : p^k \mid a_n\}$, with $p\in\mathbb P$, $k\in\mathbb N$
is a modified set of differences with at least two elements, and this implies
the following theorem.
\begin{theorem}
For an EDS in which subsequent (initial) values are co-prime the multiples
of powers of primes are equally spaced.
\end{theorem}

\section{Compagnion elliptic divisibility sequences} \label{s4}
Taking unit initial values for Somos-4,
\begin{equation} \label{iv4}
\tau_1=\tau_2=\tau_3=\tau_4=1,
\end{equation}
we have $T=1+2\alpha+\beta$, and
$\mathcal I = \alpha^2+\beta T = (\alpha+\beta)^2+\beta$.
We define \cite[Definition 1]{HS} an EDS by (\ref{eds}) and the initial values
\begin{equation}\label{ive}
a_2=-\sqrt{\alpha},\ a_3=-\beta,\ a_4=\sqrt{\alpha}\mathcal I,
\end{equation}
so $\{a_n\}$ satisfies the same recurrence as $\{\tau_n\}$.
The sequence $\{a_n\}$ is the companion EDS for Somos 4, that is, the following families of recurrences
are satisfied \cite[corollaries 1.2, 1.3]{Hone5},
\begin{equation} \label{for1}
\tau_{m+n}\tau_{m-n}=\left|
\begin{matrix}
a_n\tau_{m-1} & a_{n-1}\tau_{m} \\
a_{n+1}\tau_{m} & a_n\tau_{m+1} 
\end{matrix}\right| ,
\end{equation}
and
\begin{equation} \label{for2}
a_1a_2\tau_{m+n+1}\tau_{m-n}=\left|
\begin{matrix}
a_n\tau_{m-1} & a_{n-1}\tau_{m} \\
a_{n+2}\tau_{m+1} & a_{n+1}\tau_{m+2}
\end{matrix}\right| .
\end{equation}
Proofs of these facts can be found in \cite{Hone5,PooSwa}, cf. \cite{HS}.

To describe the companion EDS for Somos-5 with initial values
\begin{equation} \label{iv5}
\tau_1=\tau_2=\tau_3=\tau_4=\tau_5=1,
\end{equation}
we introduce an alternating sequence of functions
\[
h_{l}=\begin{cases}
2\alpha+\beta & l\equiv 0 \mod 2, \\
\alpha + 1 & l\equiv 1 \mod 2,
\end{cases}
\]
whose product equals the above mentioned invariant $\mathcal J=h_lh_{l+1}$, see
\cite[Proof of Theorem 3.7]{HS}. The companion EDS is then given by
\[
a_1=1,\ a_2=\sqrt{h_{m+n}},\ a_3=\alpha,\ a_4=-\beta a_2,
\]
and
\[
a_{k+2}a_{k-2}=h_{m+n+k}a_{k+1}a_{k-1}-\alpha a_k^2.
\]
We note that these are actually two companion sequences, one
for each value of $n+m\mod 2$, and we have
\[
\frac{a_k^2(m+n \equiv i)}{a_k^2(m+n \equiv j)}=\frac{h_i}{h_j},
\]
so any polynomial divisor of $a_k(0)$ is also a divisor of $a_k(1)$ and visa versa.
The terms of the Somos-5 sequence satisfy the same families of recurrences
(\ref{for1},\ref{for2}), see also \cite[corollary 2.12]{Hone5} and \cite[Section 7]{PooSwa}.

\section{Relative primeness} \label{s5}
Both $\alpha$ and $\beta$ are not divisors of $\tau_n$. This can be seen
by taking $\alpha=0$ or $\beta=0$. The corresponding solutions, for Somos-4,
are $\tau_n=\beta^{k_n}$ and $\tau_n=\alpha^{l_n}$, where  $k_n$ and $l_n$
satisfy the linear recurrences
\[
k_{n+2}=2k_n-k_{n-2}  + 1,\quad  l_{n+2} = l_{n+1} + l_{n-1} - l_{n-2}  + 1.
\]
In particular, they do not vanish. If $\alpha$ and $\beta$ have a divisor in common its
multiplicity will grow as (\cite[A249020]{OEIS}). Considering the terms
$\tau_n$ as polynomials in $\alpha$ and $\beta$, we can adjust the argument
of Bergman, cf. \cite{Gal}, to prove that any four consecutive terms of our polynomial
Somos-4 sequence are pairwise co-prime. The initial terms (\ref{iv4}) are co-prime. Assume that
$\{\tau_{n+i}\}_{i=-2}^1$ are pairwise co-prime, and let $p$ be an irreducible divisor of $\tau_{n+2}$ with positive degree.
Obviously $p$ does not divide $\alpha$ or $\beta$, and therefore $p\mid a_n$ if and only if
$p\mid a_{n-1}$ or $p\mid a_{n+1}$. By hypothesis this does not happen. Similarly for Somos-5
with pairwise co-prime initial values any five consecutive terms are pairwise co-prime.

\section{Divisibility in Somos-4 and Somos-5} \label{s6}

\begin{theorem}
The Somos-$k$ sequences, for $k=4,5$, with initial values $\tau_1=\cdots=\tau_k=1$ are arithmetic divisibility sequences with common difference function $d(n)=2n-k-1$.
\end{theorem}
\begin{proof}
We first assume that the initial values are co-prime and show that the set $W_q=\{n \in \mathbb Z : q \mid \tau_n\}$ is a modified set of differences. Here $\tau_n$ is a solution of either Somos-4 or Somos-5,
and we will take $q$ to be any polynomial in the ring $\mathbb Z[\alpha,\beta,\mathcal I,\tau_1^{\pm1},\tau_2,\tau_3,\tau_4]$,
or $\mathbb Z[\alpha,\beta,\mathcal J,\tau_1^{\pm1},\tau_2^{\pm1},\tau_3,\tau_4,\tau_5]$, respectively. Of course, one might take $q=\tau_k$, or let $q=p^h$ where $p$ does not occur in a denominator. Taking $m=s$ and $n=t-s$ in (\ref{for1}) gives us
\[
\tau_{t}\tau_{2s-t}=a_{t-s}^2\tau_{s-1}\tau_{s+1} - a_{t-s-1}a_{t-s+1}\tau_{s}^2.
\]
If $s,t\in W_q$ then $s\pm 1 \not\in W_q$ and hence $q\mid a_{t-s}^2$. Taking $m=s$ and $n=t-s-1$ in (\ref{for2}) yields $q\mid a_{t-s}a_{t-s-1}$, and hence $q\mid a_{t-s}$. Taking $m=s$ and $n=t-s$ in (\ref{for2}) gives us
\[
a_1a_2\tau_{t+1}\tau_{2s-t}=a_{t-s}a_{t-s+1}\tau_{s-1}\tau_{s+2} - a_{t-s-1}a_{t-s+2}\tau_{s}\tau_{s+1}.
\]
As $q\nmid a_1a_2\tau_{t+1}$ we obtain $2s-t\in W_q$ which is what we wanted to show. By lemma
\ref{lemma} it follows that $W_q$ has less than 2 elements or there exist $n,d\in\mathbb N$ such that
$m\in W_p\Rightarrow d\mid m-n$. 

We now specialise to initial values $\tau_1=\cdots=\tau_k=1$. Due to the symmetry of the recurrences, $\tau_{n+l}\leftrightarrow \tau_{n-l}$ for all $l$, we obtain $\tau_{-n+(k+1)/2}=\tau_{n+(k+1)/2}$ for all $n$.
In terms of $d=2n-k-1$ we have $\tau_{n-d}=\tau_n$.
As the degree of $\tau_k$, with $0<k<n$, is smaller than the degree of $\tau_n$, we have
$\tau_n \nmid \tau_k$ and hence $d$ is the common difference.
\end{proof}

\section{Equivalent sequences}
It is possible to choose other initial values with the same symmetry. For Somos-4, taking $\tau^\prime_1=-\tau^\prime_2=-\tau^\prime_3=\tau^\prime_4=1$ one gets the polynomial sequence
$\tau_n^\prime(\alpha,\beta)=(-1)^{\lfloor n/2 \rfloor}\tau_n(-\alpha,\beta) $, where $\tau_n(\alpha,\beta)$ is the sequence
obtained from $\tau_1=\tau_2=\tau_3=\tau_4=1$. More generally, starting from $\tau^\prime_1=\tau^\prime_4=1$
and $\tau^\prime_2=\tau^\prime_3=\gamma$ we find a sequence in $\mathbb Z[\alpha,\beta, \gamma^{\pm 1}]$, 
\begin{equation} \label{mg}
\tau_n^\prime(\alpha,\beta)=\frac{\tau_n(\gamma^3\alpha,\gamma^4\beta)}{\gamma^{(n-1)(n-4)/2}}.
\end{equation}
If $\beta=\gamma^2$ the sequence is polynomial, and for $n>4$, with $d=2n-5$,
\[
\begin{cases}
n\equiv 1 \mod 3: & \tau^\prime_n \mid \tau^\prime_m \Leftrightarrow d \mid (m-n), \\
n\not\equiv 1 \mod 3: & \tau^\prime_n \mid \tau^\prime_m \Leftrightarrow d \mid (m-n) \text{ and } \frac{m-n}{d} \not\equiv 1 \mod 3.
\end{cases}
\]
Furthermore, starting from $\tau^\prime_1=\tau^\prime_4=\delta$ and $\tau^\prime_2=\tau^\prime_3=\gamma$ we find a sequence in $\mathbb Z[\alpha,\beta, \gamma^{\pm 1},\delta^{\pm 1}]$, 
\begin{equation} \label{mgs}
\tau_n^\prime(\alpha,\beta)=\frac{\delta^{(n-2)(n-3)/2} }{\gamma^{(n-1)(n-4)/2}}\tau_n\left(\left(\frac{\gamma}{\delta}\right)^3\alpha,\left(\frac{\gamma}{\delta}\right)^4\beta\right).
\end{equation}
Both (\ref{mg}) and (\ref{mgs}) have the same divisibility properties as $\tau_n(\alpha,\beta)$.

For Somos-5, starting from $\tau^\prime_1,\ldots,\tau^\prime_5=a,b,c,b,a $ we have
\begin{equation}
\tau_n^\prime(\alpha,\beta)=\frac{a^{A_n}b^{B_n}}{c^{C_n}}\tau_n\left(\left(\frac{c}{a}\right)^2\alpha,\left(\frac{c}{a}\right)^3\beta\right),
\end{equation}
where
\[
A_n=\frac{n^2}{4}-\frac{3n}{2}+\frac{17-(-1)^n}{8},\quad B_n=\frac{1+(-1)^n}{2},\quad C_n=\frac{n^2}{4}-\frac{3n}{2}+\frac{13+3(-1)^n}{8}.
\]
We get a polynomial sequence from initial values $\tau_1=\tau_5=1$, $\tau_2=\tau_4=b$, and $\tau_3=\alpha$. Here we find the divisibility, with $d=2n-6$, $\tau_n \mid \tau_m \Leftrightarrow d \mid (m-n)$ for $n>5$, $\tau_n\mid\tau_m \Leftrightarrow n \mid m$ for $n=2,3$. 

We note the sequences $\tau$ and $\tau^\prime$ in this section are equivalent sequences in the sense of \cite[Sect. 6.3]{Swa}, and that the slightly different divisibility properties are due to initial values having a common factor.

\section{Robinson's observations} \label{s7}
We state here some observations made in \cite{Rob} and their current status.
Robinson observed that
\begin{enumerate}
\item the multiples of primes are equally spaced [such a prime is called {\em regular}]
\item the gap (the common difference) is never much larger than $p$
\item if $p$ occurs then $p^2$ occurs, and its gap is $p$ times the gap of $p$
\item if $p^i$ is the smallest occurring power, the gap of $p^{i+l}$ will be $p^l$
times the gap of $p^i$
\end{enumerate}
Much of this is now understood, but not all. Hone and Swart \cite{Hone4,Swa,HS} (as well as Naom Elkies, David Speyer, and Nelson Stephens in unpublished work, cf. \cite{Propp}) have shown that the terms $\tau_n$ of a Somos-4 sequence correspond to rational points $Q+[n]P$ on an associated elliptic curve $E$. Christine Swart, in her thesis \cite{Swa}, studies elliptic curves over $\mathbb Z_{p^r}$; due to an equivalence
\[
\tau_n \equiv 0 \mod p^r \Leftrightarrow Q+[n]P={\mathcal O}_{p^r},
\]
if $p^r$ occurs, the gap of $p^r$ equals the order $N_r$ of the (non-singular) point $P$ in $E({\mathbb Z}_{p^r})$. Swart has proved that either all powers $p^k$ are regular, or all multiples of $p$ are divisible by exactly the same power of $p$ \cite[Thm 7.6.6]{Swa}. She obtained the structure of the gap-function \cite[Thm 7.6.7]{Swa} and explained (and improved)
Robinson's bound on the gap exploiting the Hasse bound which bounds a multiple of the order of $P$, i.e. the number of points on $E(\mathbb F_p)$ within $2\sqrt{p}$ of $p+1$ \cite[Thm 7.6.5]{Swa}.  Armed with the above mentioned theorem, the gap can be calculated from the order of $P$ on $E(\mathbb F_p)$.

We conclude with a couple of examples of interest, using curves $E$ and points $P$ as given in \cite{HS}.
\begin{itemize}
\item Somos-4 with $\alpha=\beta=\tau_1=\cdots=\tau_4=1$. The prime 2 divides $\tau_m$ if and only if $5$ divides $m$. Higher  powers of 2 do not appear (the sequence mod 4 is periodic with period 10 and does not contain 0). It seems that for all
$p\neq 2$ such that $\exists n\in\mathbb N:\ p\mid \tau_n$, all powers $p^k$ are regular, and $N_r=p^{r-1}N_1$, where $N_1$
is the order of the point (1,1) on the curve $y^2=4(x^3-x)+1$ modulo $p$. 
To prove conjecture \ref{conj} it suffices to show that all powers of odd primes occurring are regular, one
does not need $w=1$ for all $p$.
\item Somos-4 with $\alpha=-1,\ \beta=2,\ (\tau_1,\cdots,\tau_4)=(1,1,2,3)$ (which extends the
Fibonacci sequence). For $p=2$ we have $w=1$ and $v=2$, for $p=3$ we have $w=1$ and $N_1=4$.
It seems that all powers of occurring primes are regular, $N_r=p^{r-w}N_1$ where $N_1$ is the order of the point $(\frac{7}{12},\sqrt{-1})$ on the singular curve
\[
y^2=\frac{(6x-5)(5+12x)^2}{216}
\]
over the field $\mathbb F_p[\sqrt{-1}]$. Again, we have not found a value of $w$ different than 1.
\item Somos-4 with $\alpha=4,\ \beta=9,\ \tau_1=\tau_4=1,\ \tau_2=\tau_3=3$.
We have $3^k\mid \tau_m \Leftrightarrow k=1,\ 3\nmid m$. Taking $p=5$ we find
$N_{r\geq3}=5^{r-3}N_1$ where $N_1=7$ is the order of the point $(55750/243, 2)$ on the curve given by
\[
y^2=4x^3-\frac{12428112196}{19683}x +\frac{1385503884676628}{14348907}
\]
over $\mathbb F_5$.
\item Somos-4 with $\alpha=2,\ \beta=5,\ (\tau_1,\ldots,\tau_4)=(1,3,2,5)$. We have
$7^k\mid \tau_m \Leftrightarrow k=2,\ N_1\mid m$, where $N_1=10$ is the order of
$P=(223081/21600, \sqrt{2})$ on
\[
y^2=4x^3-\frac{48492460561}{38880000}x +\frac{10678311547192441}{1259712000000}
\]
over $\mathbb F_7$ or over $\mathbb F_7[\sqrt{2}]$.
\end{itemize}

\section*{Acknowledgments}
I would like to thank Reinout Quispel for sharing thoughts on pleasant train rides to work, Bas Edixhoven and Jan Tuitman for useful discussions and preliminary calculations on elliptic curves during the 2014 FoCM conference, and John Roberts for suggesting the book \cite{EPSW}, which cites the thesis of Christine Swart, during the workshop on algebraic, number theoretic and graph theoretic aspects of dynamical systems at UNSW. I am grateful for travel support from AMSI, the ARC, and the La Trobe University discipline research program for mathematical \& computing sciences.

\end{document}